\input ./style/arxiv-general.cfg
\documentclass[aop,MSNbibl,seceqn,nameyear,dvips]{arximspdf}
\makeatletter
   \@ifpackageloaded{graphicx}{}{\usepackage{graphicx}}
\makeatother

\doi{10.1214/14-AOP997}
\volume{44}
\issue{2}
\pubyear{2016}
\firstpage{1107}
\lastpage{1133}
\docsubty{FLA}

\makeatletter
\newcommand{\rrvert}{\vert}
\newcommand{\rrVert}{\Vert}
\newcommand{\llvert}{\vert}
\newcommand{\llVert}{\Vert}
\renewcommand{\mid}{|}
\newcommand{\underset}[2]{\mathop{#2}\limits_{#1}}
\newcommand{\maxmin}{\operatorname{maxmin}}
\newcommand{\argmax}{\operatorname{argmax}}
\newcommand{\cupp}{\bigcup}
\newtheorem{theorem}{Theorem}[section]
\newtheorem{lemma}[theorem]{Lemma}
\newtheorem{proposition}[theorem]{Proposition}
\newproclaim{definition}[theorem]{Definition}
\newproclaim{remark}[theorem]{Remark}
\newproclaim{step}{Step}
\newproclaim{con}{Conjecture}
\makeatother

\begin{document}
\begin{frontmatter}

\title{Zero-sum repeated games: Counterexamples to the existence of
the asymptotic value and the conjecture $\operatorname{maxmin}=\operatorname{lim}v_n$}
\runtitle{Nonexistence of the asymptotic value}

\begin{aug}
\author[A]{\fnms{Bruno}~\snm{Ziliotto}\corref{}\thanksref{T1}\ead[label=e1]{ziliotto@math.cnrs.fr}}
\runauthor{B. Ziliotto}
\affiliation{GREMAQ, Universit\'{e} Toulouse 1 Capitole}
\address[A]{GREMAQ\\
Universit\'{e} Toulouse 1 Capitole\\
Manufacture des Tabacs\\
21 all\'{e}e de Brienne\\
31000 Toulouse\\
France\\
\printead{e1}}
\end{aug}
\thankstext{T1}{Supported by the ANR Jeudy (ANR-10-BLAN 0112) and the
GDR 2932.}

%
\received{\smonth{12} \syear{2013}}
%
\revised{\smonth{12} \syear{2014}}

%
\begin{abstract}
Mertens [In \textit{Proceedings of the {I}nternational {C}ongress of {M}athematicians}  (\textit{Berkeley}, \textit{Calif.}, 1986)
(1987) 1528--1577 {Amer. Math. Soc.}]
proposed two general conjectures about repeated games:
the first one is that, in any two-person zero-sum repeated game, the
asymptotic value
exists, and the second one is that, when Player~1 is more informed than
Player~2, in
the long run Player~1 is able to guarantee the asymptotic value.
We disprove these two long-standing conjectures by providing an example
of a
zero-sum repeated game with public signals and perfect observation of
the actions, where the value of the $\lambda$-discounted game does not
converge when $\lambda$ goes to 0.
The aforementioned example involves seven states, two actions and two
signals for each player. Remarkably, players observe the payoffs, and
play in turn.
\end{abstract}

%
\end{frontmatter}

\section*{Notations}
The notation ``$X:=Y$'' means ``$X$ is defined by the expression $Y$.''

The set of nonnegative integers is denoted by $\mathbb{N}$, and
$\mathbb{N}^*:=\mathbb{N} \setminus \{0 \}$. The set of
real numbers is denoted by $\mathbb{R}$.

The complementary of $B$ in $A$ is denoted by $A \setminus{B}$.

If $x \in\mathbb{R}$, the integer part of $x$ is denoted by $
\lfloor x  \rfloor$.

If $ (C,\mathcal{C} )$ is a measurable space, we denote by
$\Delta(C)$ the set of probability measures on $C$, and $\Delta_f(C)$
the set of probability measures on $C$ with finite support. We call
$\delta_c$ the Dirac measure at $c$.

If $C_0 \subset C$ is a finite set and $(\alpha_c)_{c \in C_0} \in
\Delta(C_0)$, then we write $\sum_{c \in C_0} \alpha_c \cdot c$ for
$\sum_{c \in C} \alpha_c \delta_c$.

If $X$ is a bounded real random variable, $\mathbb{E}(X)$ is the
expectation of $X$.

If $f\dvtx A \rightarrow\mathbb{R}$ is a bounded function,
$\llVert  f \rrVert  _{\infty}$ is $\sup_{x \in\mathbb
{R}} \llvert  f(x)\rrvert  $, and if $B \subset A$,
$\argmax_{x \in B} f:= \{x \in B \mid   \forall y \in
B  f(x) \geq f(y)  \}$.

If $f\dvtx  \mathbb{R} \rightarrow\mathbb{R}$ is differentiable, $\frac
{d f}{d x}$ is the derivative of $f$.

Let $a \in\mathbb{R}$ and $f$, $g$ and $h$ be real functions. If
$f(x)-g(x)$ is little-o of $h(x)$ when $x$ goes to $a$, we write $f(x)
\underset{x \rightarrow a}{=} g(x)+o(h(x))$. If $f(x)$ is equivalent
to $g(x)$ when $x$ goes to $a$, we write $f(x) \underset{x \rightarrow
a}{\sim} g(x)$.


\section*{Introduction}\label{sec1int}
The general model of two-person zero-sum repeated game was introduced
in \citet{MSZ}, Chapter IV. Such a game is described by a finite set of
states $K$, a finite set of actions $I$ (resp., $J$) for Player~1
(resp., 2), a finite set of signals $A$ (resp., $B$) for Player~1
(resp., 2), a payoff function $g\dvtx K \times I \times J \rightarrow
\mathbb{R}$, and a transition function $q\dvtx K \times I \times J
\rightarrow\Delta(K \times A \times B)$. The game proceeds as
follows. Before the game starts, a triplet $(k_1,a_0,b_0)$ is drawn
according to an initial probability distribution $p \in\Delta(K
\times A \times B)$. The state $k_1$ is the initial state, and Player~1
(resp.,~2) receives the private signal $a_0$ (resp., $b_0$).
At stage $m \geq1$, both players choose an action simultaneously and
independently, $i_m \in I$ (resp., $j_m \in J$) for Player~1 (resp.,
2). The payoff at stage $m$ is $g_m:=g(k_m,i_m,j_m)$. A triplet
$(k_{m+1},a_m,b_m)$ is drawn from the probability distribution
$q(k_m,i_m,j_m)$. The signal $a_m$ (resp., $b_m$) is announced to
Player~1 (resp., 2). The game then moves on to state $k_{m+1}$, and
enters the next stage.

For $\lambda\in(0,1]$, in the \textit{$\lambda$-discounted game},
the goal of Player~1 (resp., 2) is to maximize (resp., minimize) the
expected Abel mean of stage payoffs
$\sum_{m \geq1} \lambda(1-\lambda)^{m-1} g_m$. For $n \in\mathbb
{N}^*$, in the \textit{$n$-stage game}, the goal of Player~1 (resp.,
2) is to maximize (resp., minimize) the expected Cesaro-mean of stage
payoffs $1/n \sum_{m=1}^n g_m$. These two games have a value, denoted,
respectively, by $v_{\lambda}(p)$ and $v_n(p)$. Two important
conjectures were stated by Mertens [see \citet{M86}, page~1572 and \citet{MSZ},
Chapter~VIII, pages 378~and~386].
%

\begin{con}\label{con1}
In a zero-sum repeated game, the asymptotic value exists, that is, the
sequences of functions $(v_n)$ and $(v_{\lambda})$ converge pointwise
to the same limit, when $n$ goes to infinity and $\lambda$ goes to 0.
\end{con}

\begin{con}\label{con2}
In a zero-sum repeated game where Player~1 is more informed than Player~2 (i.e., Player~1 observes what Player~2 observes), $(v_n)$ and
$(v_{\lambda})$ converge pointwise to the $\maxmin$ of the game,
which is the maximal amount that Player~1 can guarantee to herself in
long games.
\end{con}

The Mertens' conjectures have been proven true in numerous particular
classes of zero-sum repeated games [see
\citet{AM95,BK76,GOV13,MN81,MZ71,N08},
\citeauthor{R06} (\citeyear{R06,R12}),
\citet{rosenberg00},
\citeauthor{RSV02} (\citeyear{RSV02,RSV03,RSV04}),
\citet{RV00},
\citeauthor{sorin841} (\citeyear{sorin841,sorin852})].

In zero-sum repeated games with symmetric information [the particular
class of zero-sum repeated games in which players observe the actions
perfectly and receive a public signal about the state, i.e., for all
$m \in\mathbb{N}^*$, $a_m=b_m$ and $(i_m,j_m)$ is $a_m$-measurable],
Conjectures \ref{con1} and \ref{con2} were not proven [see \citet{F82,KZ74,V10} for a
proof of the conjectures under additional assumptions on the transition
function]. Note that in this class of games, Player~1 has the same
information as Player~2, and in particular she is more informed than
Player~2.

The main contribution of this paper is to provide an example of a
zero-sum repeated game with symmetric information where $(v_{\lambda
})$ does not converge when $\lambda$ goes to $0$, thus contradicting
both conjectures.
Our example also shows that there is no hope of obtaining an existence
result for the asymptotic value in zero-sum repeated games, unless we
make a very strong assumption on the transition function.
Indeed, the structure of the example is very simple: at each stage, the
action of one player only influences their stage payoff and the
transition; moreover, players observe the payoff. In addition, since
the information is symmetric, the belief hierarchy [belief of Player~1
about the state or actions played, belief of Player~2 about the belief
of Player~1 about the state or actions played, and so on; see \citet
{MZ85} for more details] is straightforward: players know the past
actions and have the same belief about the state.

In addition, using the example we will construct a second example of a
zero-sum repeated game with symmetric information where neither
$(v_{\lambda})$ nor $(v_n)$ converge.

Lastly, we will provide an example of a state-blind zero-sum repeated
game (the particular class of zero-sum repeated games with symmetric
information where players get no signals about the state) and an
example of a zero-sum repeated game with one state-blind player (one
player observes the state but the other gets no signal about it)
without an asymptotic value. We will also give an example of a standard
zero-sum stochastic game with compact action sets without an asymptotic
value, providing an alternative counterexample to \citet{vigeral13}.
Note that this last class of zero-sum repeated games does not concern
the two conjectures, since the action sets are not finite. Nonetheless,
the example is interesting because it has a simpler structure than the
one given in \citet{vigeral13} (players play in turn) and is thus
easier to analyze.




The paper is organized as follows. In Section~\ref{sec1}, we explain the model
of zero-sum repeated game with symmetric information and some basic
concepts. In Section~\ref{sec2}, we present our main counterexample and show
that $(v_{\lambda})$ does not converge. In Section~\ref{sec3}, we construct a
similar game where neither $(v_{\lambda})$ nor $(v_n)$ converge.
In Section~\ref{sec4}, we show how our counterexample adapts to other classes of
zero-sum repeated games.

\section{Zero-sum repeated games with symmetric information}\label{sec1} \label{gen}

\subsection{The model}\label{sec1.1} \label{genmodel}

A zero-sum repeated game with symmetric information $\Gamma$ is
defined by the following elements:
\begin{itemize}
\item[--]
State space $K$.
\item[--]
Action set $I$ (resp., $J$) for Player~1 (resp., 2).
\item[--]
Signal set $A$.
\item[--]
Transition function $q\dvtx K \times I \times J \rightarrow\Delta(K \times A)$.
\item[--]
Payoff function $g\dvtx K\times I \times J \rightarrow\mathbb{R}$.\vadjust{\goodbreak}
\end{itemize}
We assume $K$, $I$, $J$ and $A$ to be finite. Both players know $K,  I,  J,  A,  g,  q$.

Given an initial probability $p \in\Delta(K)$ known by both players,
the game $\Gamma^{p}$ proceeds as described below:
\begin{itemize}
\item[--]
Before the game starts, an initial state $k_1$ is drawn according to
$p$. The quantity $k_1$ is the initial state. Players do not know $k_1$.
\item[--]
At stage $m \geq1$, both players choose an action simultaneously and
independently, $i_m \in I$ (resp., $j_m \in J$) for Player~1 (resp.,
2). The payoff at stage $m$ is $g(k_m,i_m,j_m)$. A pair $(k_{m+1},a_m)$
is drawn from $q(k_m,i_m,j_m)$. Both players receive the public signal
$a_m$, which contains the actions $i_m$ and $j_m$. The game moves on to
state $k_{m+1}$, and then continues to the next stage.
\end{itemize}
Compared to the model of general repeated game described in the\break 
\hyperref[sec1int]{Introduction}, when the game has symmetric information, one has
$a_m=b_m$ and $(i_m,j_m)$ is \mbox{$a_m$-}measurable for all $m \geq1$, and
the players do not receive a signal at the outset of the game.
The \textit{history of the game before stage $m$} is the random
sequence $(k_1,i_1,j_1,a_1,k_2,\ldots,i_{m-1},j_{m-1},a_{m-1},k_m)$.

The set of all possible histories before stage $m$ is
$H_m:= K \times(I \times J \times A \times K)^{m-1}$.

The set of all possible plays is
$H_\infty:= K \times(I \times J \times A \times
K)^{\mathbb{N}^*}$.

At the beginning of stage $m$, both players know $(i_1,j_1,a_1,
\ldots,i_{m-1},j_{m-1}, \break a_{m-1})$.

A \textit{pure strategy} for Player~1 (resp., 2) is a map $s\dvtx \cupp_{m
\geq1} (I \times J \times A)^{m-1} \rightarrow I$ [resp., $t\dvtx \cupp_{m
\geq1} (I \times J \times A)^{m-1} \rightarrow J$].

A \textit{behavior strategy} for Player~1 (resp., 2) is a map
$\sigma\dvtx \cupp_{m \geq1} (I \times J \times A)^{m-1} \rightarrow
\Delta(I)$ [resp., $\tau\dvtx \cupp_{m \geq1} (I \times J \times
A)^{m-1} \rightarrow\Delta(J)$]. The set of all behavior strategies
for Player~1 (resp., 2) is denoted $\Sigma$ (resp., $\mathcal{T}$).

An initial probability $p \in\Delta(K)$ and a pair of (pure or
behavior) strategies $(\sigma,\tau) \in\Sigma\times\mathcal{T}$
naturally induce a probability measure $\mathbb{P}^{p}_{\sigma,\tau
}$ on the set of all finite histories $\cupp_{m \geq1} H_m$ [for more
details, see \citet{sorin02b}, Appendix~D]. By the Kolmogorov extension
theorem, this probability measure uniquely extends to $H_{\infty}$.
We denote $\mathbb{E}^{p}_{\sigma,\tau}$ the expectation with
respect to the probability measure $\mathbb{P}^{p}_{\sigma,\tau}$.
Let $g_m$ be the random payoff at stage $m \geq1$: $g_m:=g(k_m,i_m,j_m)$.

For $\lambda\in(0,1]$, the \textit{$\lambda$-discounted game} is
the strategic-form game $\Gamma^{p}_{\lambda}$ with strategy set
$\Sigma$ for Player~1 and $\mathcal{T}$ for Player~2, and payoff
function $\gamma^{p}_{\lambda}\dvtx \Sigma\times\mathcal{T} \rightarrow
\mathbb{R}$ defined by
\[
\gamma^{p}_{\lambda}(\sigma,\tau):=\mathbb{E}^{p}_{\sigma,\tau
}
\biggl(\sum_{m \geq1} \lambda(1-\lambda)^{m-1}
g_m \biggr) \cdot
\]
The goal of Player~1 (resp., 2) is to maximize (resp., minimize)
$\gamma^{p}_{\lambda}$.

For $n \in\mathbb{N}^*$, the \textit{$n$-stage repeated game} is the
strategic-form game $\Gamma^{p}_n$ with strategy set $\Sigma$ for
Player~1 and $\mathcal{T}$ for Player~2, and payoff function $\gamma
^{p}_{n}\dvtx \Sigma\times\mathcal{T} \rightarrow\mathbb{R}$ defined by
\[
\gamma^{p}_n(\sigma,\tau):=\mathbb{E}^{p}_{\sigma,\tau}
\Biggl(\frac{1}{n} \sum_{m=1}^n
g_m \Biggr) \cdot
\]
The\vspace*{1pt} goal of Player~1 (resp., 2) is to maximize (resp., minimize)
$\gamma^{p}_{n}$.

The games $\Gamma^{p}_{\lambda}$ and $\Gamma^{p}_n$ have a value,
denoted, respectively, by
$v_{\lambda}(p)$ and $v_n(p)$ [see \citet{MSZ}, Chapter IV]. That is,
there are real numbers $v_{\lambda}(p)$ and $v_{n}(p)$ satisfying
%
\begin{eqnarray*}
v_{\lambda}(p)&=&\max_{\sigma\in\Sigma} \min_{\tau\in\mathcal
{T}}\gamma^{p}_{\lambda}(\sigma,\tau)=\min_{\tau\in\mathcal
{T}}\max_{\sigma\in\Sigma}\gamma^{p}_{\lambda}(\sigma,\tau),
\\
v_n(p)&=&\max_{\sigma\in\Sigma} \min_{\tau\in\mathcal{T}}
\gamma ^{p}_n(\sigma,\tau)=\min_{\tau\in\mathcal{T}}
\max_{\sigma\in
\Sigma}\gamma^{p}_{n}(\sigma,\tau).
\end{eqnarray*}
In the game $\Gamma^{p}_{\lambda}$, a strategy $\sigma^* \in\Sigma
$ (resp., $\tau^* \in\mathcal{T}$) is \textit{optimal} if for all
$\tau\in\mathcal{T}$ (resp., $\sigma\in\Sigma$) we have $\gamma
_{\lambda}(\sigma^*,\tau) \geq v_{\lambda}(p)$ [resp., $\gamma
_{\lambda}(\sigma,\tau^*) \leq v_{\lambda}(p)$]. Optimal strategies
in $\Gamma^{p}_n$ are defined in the same way, replacing $\lambda$ by $n$.

\subsection{Asymptotic approach}\label{sec1.2}

\begin{definition}
$\Gamma$ has an \textit{asymptotic value} if the sequences of
functions $(v_n)$ and $(v_{\lambda})$ converge pointwise to the same
limit (when $n \rightarrow+\infty$ and $\lambda\rightarrow0$).
\end{definition}

%
\begin{remark}
For all $(n,\lambda) \in\mathbb{N}^* \times(0,1]$, $v_n\dvtx \Delta(K)
\rightarrow\mathbb{R}$ and
$v_{\lambda}\dvtx \Delta(K) \rightarrow\mathbb{R}$ are $\llVert  g\rrVert  _{\infty}$-Lipschitz [see
\citet{MSZ}, Chapter V, page~184]. Thus, as
far as these sequences are concerned, pointwise and uniform convergence
are equivalent. In what follows,
we will simply write $``(v_n)$ converges'' or $``(v_{\lambda})$
converges,'' whenever these sequences of functions converge pointwise.
\end{remark}

\subsection{Equivalent repeated game with perfect observation of the state}\label{sec1.3} \label{repu}

Let $\Gamma$ be a repeated game with symmetric information. For $m
\geq1$, we denote $p_m$ the conditional probability on the state $k_m$
at stage $m$, given the random past history of the players
$(i_1,j_1,a_1,\ldots,i_{m-1},j_{m-1},a_{m-1})$. 
The random variable $p_m$ represents the common belief at stage $m$
about the current state $k_m$. The triplet $(p_{m+1},i_m,j_m)$ is the
only relevant information conveyed by the signal $a_m$, and $p_m$ plays
the role of a state variable. Indeed, extend $q$ and $g$ linearly to
$\Delta(K) \times I \times J$:
$q(p,i,j):=\sum_{k \in K} p(k) q(k,i,j)$ and $g(p,i,j):=\sum_{k \in
K} p(k) g(k,i,j)$.

Assume that at some stage of the game, players have a common belief $p$
about the current state. If they play $(i,j)$ and receive the signal $a
\in A$, then their posterior belief about the next state will be
$q_{\Delta(K)\mid A}(p,i,j,a) \in\Delta(K)$, where
\[
\forall k \in K \qquad q_{\Delta(K)\mid A}(p,i,j,a) (k):= \frac{q(p,i,j)(k,a)}{\sum_{k' \in K} q(p,i,j)(k',a)}.
\]
Let $q_{A}(p,i,j) \in\Delta(A)$ be the marginal on $A$ of $q(p,i,j)$.
We define $\widetilde{q}\dvtx \Delta(K) \times I \times J \rightarrow
\Delta_f(\Delta(K))$ by
%
%
\begin{equation}
\widetilde{q}(p,i,j):=\sum_{a \in A}
q_A(p,i,j) (a) \cdot q_{\Delta
(K)\mid A}(p,i,j,a).
\end{equation}
%
If players have a common belief $p$ about the current state and play
$(i,j)$, then for all $a \in A$, their posterior belief about the next
state will be $q_{\Delta(K)\mid A}(p,i,j,a)$ with probability $q_A(p,i,j)(a)$.

Fix $\lambda\in(0,1]$. The function $v_{\lambda}\dvtx  \Delta(K)
\rightarrow\mathbb{R}$ is the unique solution of the following
functional equation [see \citet{MSZ}, Chapter IV, Theorem~3.2, page~158]:
%
%
\begin{eqnarray}
\label{shapley} f(p)&=&\max_{x \in\Delta(I)} \min_{y \in\Delta(J)}
\bigl\{ \lambda{g}(p,x,y)+(1-\lambda)\mathbb{E}^p_{x,y}(f)
\bigr\}
\\
\label{shapley2} &=&\min_{y \in\Delta(J)} \max_{x \in\Delta(I)}
\bigl\{\lambda {g}(p,x,y)+(1-\lambda)\mathbb{E}^p_{x,y}(f)
\bigr\},
\end{eqnarray}
where the unknown is a continuous function $f\dvtx \Delta(K) \rightarrow
\mathbb{R}$,
\[
\mathbb{E}^p_{x,y}(f):=\sum_{(p',i,j) \in\Delta(K) \times I \times
J}
x(i)y(j)\widetilde{q}(p,i,j) \bigl(p'\bigr)f\bigl(p'
\bigr)
\]
and
\[
g(p,x,y):=\sum_{(i,j) \in I \times J} x(i)y(j){g}(p,i,j).
\]
The game that is equivalent to $\Gamma$
is the repeated game $\widetilde{\Gamma}$ with state space $\Delta
(K)$, action sets $I$ and $J$, transition function $\widetilde
{q}\dvtx \Delta(K) \times I \times J \rightarrow\Delta_f (\Delta
(K) )$ 
and payoff function $g$.

Given an initial state $p \in\Delta(K)$, the game $\widetilde{\Gamma
}{}^{p}$ proceeds as follows. Players know $p$, and at each stage $m \geq
1$, they choose an action simultaneously and independently, $i_m \in I$
(resp., $j_m \in J$) for Player~1 (resp., 2). The payoff at stage $m$
is $g(p_m,i_m,j_m)$, and $p_{m+1}$ is drawn from $\widetilde
{q}(p_m,i_m,j_m)$, and announced to both players. The game moves on to
state $p_{m+1}$, and then continues to the next stage. Note that $p_m$
can only take a countable number of values.

Given $\lambda\in(0,1]$ and $n \in\mathbb{N}^*$, the $\lambda
$-discounted game $\widetilde{\Gamma}{}^{p}_{\lambda}$ and the
$n$-stage repeated game $\widetilde{\Gamma}{}^{p}_n$ are defined as in
Section~\ref{genmodel}.
The games $\widetilde{\Gamma}{}^{p}_{\lambda}$ (resp., $\widetilde
{\Gamma}{}^{p}_n$) and $\Gamma^{p}_{\lambda}$ (resp., $\Gamma^{p}_n$)
have the same value, and optimal strategies in the first game induce
optimal strategies in the second one, and vice versa.

\begin{definition}
A strategy in $\widetilde{\Gamma}$ is \textit{stationary} if it only
depends on the state variable $p_m$. Such a strategy can be seen as a
map from $\Delta(K)$ to $\Delta(I)$ or $\Delta(J)$.
\end{definition}

There exists stationary strategies which are optimal in
$\widetilde{\Gamma}{}^{p}_{\lambda}$ for any $p \in\Delta(K)$ [see
\citet{MSZ}, Chapter VII, Proposition 1.4, page~326]. We have the
following refinement, which follows from \citet{SH53} and the
compactness of $\Delta(K)$:

\begin{definition}
A player is said to \textit{control} $p \in\Delta(K)$ if in this
state the transition $\widetilde{q}(p,\cdot)$ and the payoff ${g}(p,\cdot)$ do
not depend on the action of the other player.
\end{definition}

\begin{lemma} \label{pure}
Assume that each state in $\Delta(K)$ is controlled by one player.
Then both players have pure stationary strategies which are optimal in
$\widetilde{\Gamma}{}^p_{\lambda}$ for any $p \in\Delta(K)$.
\end{lemma}


\section{A repeated game with symmetric information where \texorpdfstring{$(v_{\lambda})$}{(v{lambda})} does not converge}\label{sec2} \label{counter}
First we present the main counterexample of the paper. We then describe
the equivalent game with perfect observation of the state and actions.
This game might seem intricate, but it turns out that for each discount
factor $\lambda$ in $(0,1]$, the discounted game is equivalent to a
strategic-form game with strategy sets $\mathbb{N}$ for Player~1 and
$2 \mathbb{N}$ for Player~2. From the analysis of this last game, we
deduce that the discounted value of the main counterexample does not converge.

\subsection{Description of the example}\label{sec2.1} \label{describe}
Consider the following repeated game with symmetric information $\Gamma
$, with state space
$K= \{1^*,1^{++},1^T,1^{+},0^*,\break 0^{++},0^+  \}$, action sets
$I=J= \{ C,Q  \}$, and signal set $A= \{D,D'  \}
$. 
The payoff function does not depend on the actions, and is equal to 1
in states $1^*$, $1^{++}$, $1^T$ and $1^+$, and to 0 in states $0^*$,
$0^{++}$ and $0^+$.
Player~2 controls states $1^{++}$, $1^T$ and $1^+$.
Player~1 controls states $0^{++}$ and $0^+$.
Lastly, states $1^*$ and $0^*$ are \textit{absorbing} states: once
$1^*$ or $0^*$ is reached, the game remains forever in this state, and
the payoff does not depend on the actions (\textit{absorbing} payoff).
Figure~\ref{fig1} describes the transition function.

We have adopted the following notation: an arrow going from state $k
\in K$ to state $k' \in K$ with the caption $(i,p,a) \in \{C,Q
 \} \times[0,1] \times \{D,D'  \}$ indicates that if
the player who controls state $k$ plays action $i$, then with
probability $p$ the state moves to state $k'$ and the signal is $a$.
For example, if the state is $1^{++}$ and Player~2 plays action $C$,
then with probability $1/2$ the game moves to state $1^T$ and the
signal is $D$, and with probability $1/2$ the game stays in state
$1^{++}$ and the signal is $D'$.

%
\begin{figure}[t]

\includegraphics{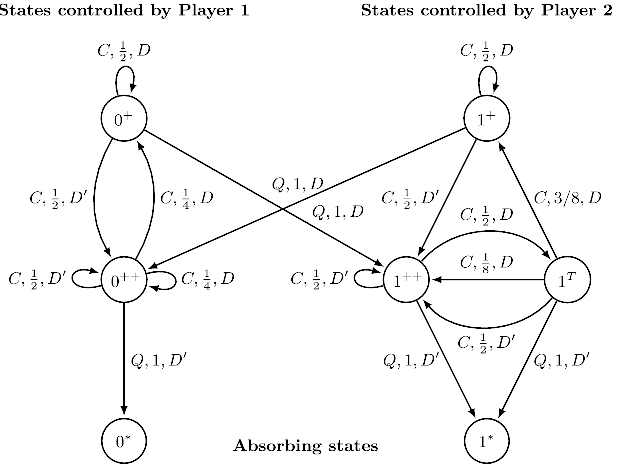}

\caption{Transitions in the game $\Gamma$.}\label{fig1}
\end{figure}

The action $Q$ causes absorption or switching from $ \{
1^{++},1^{T},1^+  \}$ to
$ \{0^{++},\break 0^+  \}$ and vice versa. In particular, the
players know in which of the following subsets of $K$ the current state
is: $ \{1^{++},1^T,1^+ \}$, $ \{0^{++},0^+  \}$,
$ \{1^*  \}$ or $ \{0^*  \}$.

\subsection{Equivalent repeated game with perfect observation}\label{sec2.2} \label{red}
In this subsection, we give the exact expression of the transition
function $\widetilde{q}$ of the\vspace*{1pt} equivalent repeated game with perfect
observation of the state and actions $\widetilde{\Gamma}$. We denote
by $1_{2n}$, $1_{2n+1}$, $0_n \in\Delta(K)$ ($n \in\mathbb{N}^*$)
the possible beliefs of the players along the game. Starting from the
prior belief $0^{++}$, $0_n$ is the belief after $n$ consecutive stages
in which Player~1 played $C$ and the signal was $D$; starting from the
prior belief $1^{++}$, $1_{n}$ is the belief after $n$ consecutive
stages in which Player~2 played $C$ and the signal was~$D$.

Formally, given $n \in\mathbb{N}^*$, we define the beliefs
$1_{2n},1_{2n+1},0_n \in\Delta(K)$ by:
\begin{eqnarray*}
1_{2n}&:=&2^{-2n} \cdot1^{++}+\bigl(1-2^{-2n}
\bigr) \cdot1^+,
\\
1_{2n+1}&:=&2^{-2n} \cdot1^{T}+\bigl(1-2^{-2n}
\bigr) \cdot1^+,
\\
0_n&:=&2^{-n} \cdot0^{++}+\bigl(1-2^{-n}
\bigr)\cdot0^+.
\end{eqnarray*}
Let us suppose that at some stage of $\Gamma$, the belief of the
players about the current state is $1_{n}$, for some $n \in\mathbb
{N}$. Player~1's action has no influence on the transition and both
players know it. If Player~2 plays $C$, then with probability $1/2$ the
signal is~$D$ (resp., $D'$), and by Bayes rule the posterior belief
about the next state will be $1_{n+1}$ (resp., $1_0=1^{++}$). Now\vspace*{1pt} let
us suppose that the belief is $1_{2n}$ or $1_{2n+1}$. If he plays $Q$,
then with probability $1-2^{-2n}$ (resp., $2^{-2n}$) the signal is $D$
(resp., $D'$), and the posterior belief about the next state will be
$0^{++}$ (resp., $1^*$). Thus, in $\widetilde{\Gamma}$, the
transition function $\widetilde{q}$ in the states $1_{2n}$ and
$1_{2n+1}$ can be described by Figure~\ref{fig2}.
%
\begin{figure}[t]

\includegraphics{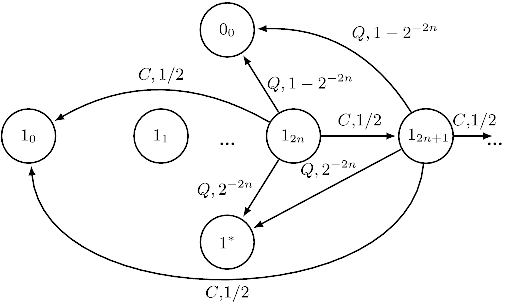}

\caption{Transitions in the states $1_{2n}$ and $1_{2n+1}$.}
\label{fig2}
%
%
%
%
%
\end{figure}



Let us suppose that at some stage of $\Gamma$, the belief of the
players about the current state is $0_{n}$. Player~2's action has no
influence on the transition and both players know it. If Player~1 plays
$C$, then with probability $1/2$ the signal is $D$ (resp., $D'$), and
by Bayes rule the posterior belief about the next state will be
$0_{n+1}$ (resp., $0_0=0^{++}$). Now let us suppose that the belief is
$0_n$. If she plays $Q$, then with probability $1-2^{-n}$ (resp.,
$2^{-n}$) the signal is $D$ (resp., $D'$), and the posterior belief
about the next state will be $1^{++}$ (resp., $0^*$). Thus, in
$\widetilde{\Gamma}$, the transition function $\widetilde{q}$ in the
state $0_{n}$ can be described by Figure~\ref{fig3}.
%
\begin{figure}

\includegraphics{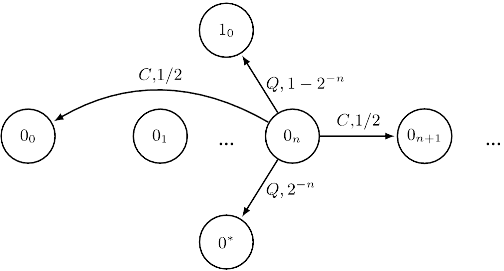}

\caption{Transitions in the state $0_n$.}
\label{fig3}
%
\end{figure}

Let $P_1:=\cupp_{n \in\mathbb{N}}  \{1_{n}
\}$, $P_0:=\cupp_{n \in\mathbb{N}}  \{0_{n}
\}$, and $P=P_1 \cup P_0 \cup \{1^*,0^* \}$.
Note that in $\widetilde{\Gamma}$, Player~1 controls all the states
in $P_0$, and Player~2 controls all the states in~$P_1$. The set of
states which can be reached with positive probability under some
strategy vector in $\widetilde{\Gamma}{}^{1^{++}}$ is exactly $P$.

Now let us explain the dynamics of the game informally. Assume that the
game starts in state $p=0_0=0^{++}$. Since the payoff is 1 in states
lying in $P_1$ and $0$ in states lying in $P_0$, and since Player~1
maximizes the payoff, Player~1 wants to go to state $1_0=1^{++}$. If
she plays $Q$ immediately, the game is absorbed in state $0^*$, which
is the worst state for her. If she never plays $Q$, the payoff is 0
forever, which is also an unfavorable outcome for her. If she plays $C$
until the state is $0_{n}$, and then she plays $Q$, then the game is
absorbed in state $0^*$ with probability $2^{-n}$ (we will often call
\textit{absorbing risk taken by Player}~1 the probability that the
game is absorbed in state $0^*$), and the game goes to state $1^{++}$
with a probability of $(1-2^{-n})$.

To reach state $0_{n}$ from state $0^{++}$, Player~1 needs $2^{n}$
stages on average. Hence, Player~1 has to make a trade-off between
staying not too long in states of type 0, and having a low probability
of being absorbed in $0^*$ when she plays $Q$. Basically, Player~1
needs to wait on average $2^n$ stages to reduce the absorbing risk to $2^{-n}$.

The same principle applies to Player~2. Assume that the game starts in
state $p=1_0=1^{++}$. Player~2 plays $C$ until reaching state $1_{2n}$,
and then $Q$. The game is absorbed in state $1^*$ with a probability of
$2^{-2n}$, and it goes to state state $0^{++}$ with a probability of
$(1-2^{-2n})$.

To reach state $1_{2n}$, Player~2 needs on average $2^{2n}$ stages.
Player~2 can also play $Q$ in state $1_{2n+1}$, but this is not a good
strategy, since such a state is harder to reach than state $1_{2n}$
($2^{2n+1}$ stages on average) but leads to the same absorbing risk
$2^{-2n}$. Note that the time needed by Player~2 to go from state
$1^{++}$ to state $1_{2n}$ is on average the same as the time needed by
Player~1 to go from state $0^{++}$ to state~$0_{2n}$.
The state of the game oscillates between states of type 0 and states of
type 1 as long as it does not reach an absorbing state. The only
asymmetry of the game is that Player~1 can take any absorbing risks of
the form $2^{-n}$, whereas Player~2 can only take absorbing risks of
the form $2^{-2n}$.

\subsection{Equivalent strategic-form game}\label{sec2.3} \label{oneshot}
Fix $\lambda\in(0,1]$.
Let $(a,b) \in\mathbb{N} \times2 \mathbb{N}$.
Let $s(a) \in\Sigma$ be the following pure stationary strategy for
Player~1: for every $n \geq a$, play $Q$ in state $0_n$; otherwise play $C$.

Let $t(b) \in\mathcal{T}$ be the following pure stationary strategy
for Player~2: for every $n \geq b$, play $Q$ in state $1_n$; otherwise
play $C$.

The aim of this section is to prove the following proposition.

\begin{proposition} \label{eq2}
The game $\widetilde{\Gamma}_{\lambda}^{p}$ for $p=1^{++}$ has the
same value as the strategic-form game $G_{\lambda}$ with action set
$\mathbb{N}$ for Player~1, $2 \mathbb{N}$ for Player~2, and payoff function
%
%
\begin{equation}
\label{discpayoff} g_{\lambda}(a,b):=\frac{1-f_{\lambda}(b)}{1-f_{\lambda
}(a)f_{\lambda}(b)},
\end{equation}
where
\[
f_{\lambda}(n):=\frac{(1-2^{-n})(1-\lambda^2)}{1+2^{n+1}\lambda
(1-\lambda)^{-n}-\lambda}.
\]
Moreover, $(a,b) \in\mathbb{N} \times2 \mathbb{N}$ are optimal
strategies in $G_{\lambda}$ if and only if $s(a)$ and $t(b)$ are
optimal strategies in $\widetilde{\Gamma}_{\lambda}^{1^{++}}$.
\end{proposition}

\begin{pf}
First note that by Lemma \ref{pure}, there exists pure optimal
stationary strategies in $\widetilde{\Gamma}{}^{1^{++}}_{\lambda}$. A
pure stationary strategy for Player~1 corresponds to a strategy $s(a)$,
where $a \in\mathbb{N}$ is the smallest integer for which Player~1
plays $Q$ in state $1_a$. Note that Player~2 is better off quitting in
state $1_{2n}$ rather than in state $1_{2n+1}$. Indeed, state
$1_{2n+1}$ is harder to reach than state $1_{2n}$, but the probability
of being absorbed in state $1^*$ when playing $Q$ is the same in both
states. Thus, Player~2 has an optimal strategy of the form $t(b)$, for
some $b \in2 \mathbb{N}$. Fix now $a,b \in\mathbb{N} \times2
\mathbb{N}$, and let us compute the payoff $\gamma^{1^{++}}_{\lambda
}(s(a),t(b))$ given by these strategies.


Let $T_b=\inf \{m \geq1\mid j_m=Q \}$ and $T_a=\inf \{m
\geq T_b+1\mid i_m=Q \}-T_b$.
The random variable $T_b$ (resp., $T_a$) represents the time spent by
Player~2 (resp., Player~1) in states of type $1$ (resp., $0$) before quitting.
The payoff $\gamma_\lambda^{1^{++}}(s(a),\break t(b))$ is equal to
%
\begin{eqnarray*}
&& 2^{-b}+\bigl(1-2^{-b}\bigr)
\mathbb{E} \Biggl(\sum_{m=1}^{T_b} \lambda
(1-\lambda)^{m-1}1 + \sum_{m=T_b+1}^{T_a+T_b}
\lambda(1-\lambda)^{m-1}0 \Biggr)
\\
&&\qquad{} +\bigl(1-2^{-b}\bigr) \bigl(1-2^{-a}\bigr) \bigl((1-
\lambda)^{T_a+T_b} \bigr) \gamma _{\lambda}^{1^{++}}
\bigl(s(a),t(b)\bigr),
\end{eqnarray*}
where $\mathbb{E}^{1^{++}}_{s(a),t(b)}$ is denoted by $\mathbb{E}$.
The quantity $2^{-b}$ corresponds to the probability that the game is
absorbed in state $1^*$ when Player~2 plays $Q$: in this case the
payoff is $1$ at every stage. If the game is not absorbed at that
point, then the payoff from stage 1 until the stage when Player~1 plays
$Q$ is the second term of the equation. When Player~1 plays $Q$, the
game is absorbed in state $0^*$ with probability $2^{-a}$, and goes
back to state $1^{++}$ with probability $(1-2^{-a})$; this is the third
term of the equation.

We deduce that
%
%
\begin{equation}
\label{step1} \gamma_{\lambda}^{1^{++}}\bigl(s(a),t(b)\bigr)=
\frac{1-(1-2^{-b})\mathbb
{E} ((1-\lambda)^{T_b}  )}{1-(1-2^{-a})(1-2^{-b})\mathbb
{E} ((1-\lambda)^{T_a+T_b} )}.
\end{equation}
Under $s(a)$ and $t(b)$, the stopping times $T_a$ and $T_b$ are
independent. Thus, we have
%
%
\begin{equation}
\label{step2} \mathbb{E} \bigl((1-\lambda)^{T_a+T_b} \bigr)=\mathbb{E}
\bigl((1-\lambda)^{T_a} \bigr) \mathbb{E} \bigl((1-\lambda)^{T_b}
\bigr) \cdot
\end{equation}
%
Under $s(a)$ [resp., $t(b)$], $T_a-1$ (resp., $T_b-1$) is a random
variable denoting the number of trials needed to have $a$ (resp., $b$)
consecutive successes in independent trials with a success probability
of $1/2$. Thus, this random variable follows the generalized geometric
distribution of order $a$ (resp., $b$) and parameter $1/2$ studied in
\citet{PGP83}. By Lemma~2.2 in \citet{PGP83}, we get
%
%
\begin{equation}
\label{step3} \mathbb{E} \bigl((1-\lambda)^{T_n} \bigr)=(1+\lambda)/
\bigl(1+2^{n+1} \lambda(1-\lambda)^{-n}-\lambda\bigr).
\end{equation}
Combining (\ref{step1}), (\ref{step2}) and (\ref{step3}), we get the
desired result.
\end{pf}
%
\subsection{Asymptotic study of \texorpdfstring{$G_{\lambda}$}{G{lambda}} and proof of the main theorem}\label{sec2.4} \label{limit}
We first determine optimal strategies in $G_{\lambda}$.

\begin{proposition} \label{stratopt}
Let $(a^*,b^*) \in\argmax_{n \in\mathbb{N}}
f_{\lambda} \times\argmax_{n \in2\mathbb{N}} f_{\lambda}$. Then
$a^*$ (resp., $b^*$) is a dominant strategy for Player~1 (resp., 2) in
$G_{\lambda}$. In particular, they are optimal strategies in
$\widetilde{\Gamma}{}^{1^{++}}_{\lambda}$.
\end{proposition}

\begin{pf}
We have $ \lim_{n \rightarrow+\infty} f_{\lambda
}(n)=0$, therefore, $a^*$ and $b^*$ are well defined.
Observe that the function $ (x,y) \rightarrow
(1-y)(1-xy)^{-1}$, defined on $[0,1)^2$, is increasing in $x$ for every
fixed $y$, and decreasing in $y$ for every fixed $x$, and that
$f_\lambda(\mathbb{N}) \subset[0,1)$. Hence, $a^*$ and $b^*$ are
dominant strategies in $G_{\lambda}$, and by Proposition \ref{eq2}
they are optimal strategies in $\widetilde{\Gamma}{}^{1^{++}}_{\lambda}$.
\end{pf}
%

\begin{remark}
The existence of dominant strategies in $G_{\lambda}$ can be explained
without any computation. Indeed, in $\widetilde{\Gamma
}{}^{1^{++}}_{\lambda}$, the maximization problem faced by Player~2 does
not depend on the payoff he receives once reaching state $1_0$, as long
as this payoff is positive. Therefore, whichever stationary strategy
Player~1 chooses, the best-response for Player~2 is always the same.
The same argument applies to Player~1.
\end{remark}

To study $f_{\lambda}$, it is convenient to make the change of
variables $r=2^{-n}$.
Let $\widehat{f}_{\lambda}\dvtx [0,1] \rightarrow\mathbb{R}$ be the
function defined for $r \in(0,1]$ by
\[
\widehat{f}_{\lambda}(r):=(1-r) \bigl(1+2 \lambda r^{-s}-\lambda
\bigr)^{-1},
\]
where $  s:=1-\ln(2)^{-1} \ln(1-\lambda)>1$, and
$\widehat{f}_{\lambda}(0):=0$.
Note that for all $n \in\mathbb{N}$, $  f_{\lambda
}(n)=(1-\lambda^2)\widehat{f}_{\lambda}(2^{-n})$.

\begin{lemma} \label{maxf}
The function $\widehat{f}_{\lambda}$ reaches its maximum at a unique
point $r^*(\lambda)$, is strictly increasing on $ [0,r^*(\lambda
)  ]$, and strictly decreasing on $ [r^*(\lambda),1
]$. Moreover, for all $c>0$,
%
%
\begin{equation}
\widehat{f}_{\lambda}(c \sqrt{2\lambda})\underset{\lambda \rightarrow0}
{=}1-\bigl(c+c^{-1}\bigr)\sqrt{2\lambda}+o(\sqrt{\lambda})
\end{equation}
and
%
%
\begin{equation}
r^*(\lambda) \underset{\lambda\rightarrow0} {\sim}\sqrt{2 \lambda }.
\end{equation}
\end{lemma}

\begin{pf}
Differentiating $\widehat{f}_{\lambda}$ yields
\[
\widehat{f}_{\lambda}'(r)=\frac{-(1+2 \lambda r^{-s}-\lambda
)-(1-r)(-2\lambda s r^{-s-1})}{(1+2 \lambda r^{-s}-\lambda)^2}.
\]
The numerator of this expression is equal to $h_{\lambda}(r):=\lambda
-1+2\lambda(-(1+s)r+s)r^{-s-1}$. Note that
\[
h_{\lambda}'(r)=-2\lambda\bigl(s(1+s) (1-r)
\bigr)r^{-s-2}.
\]
We have $h_{\lambda}'<0$ on $(0,1)$, $ \lim_{r
\rightarrow0} h_{\lambda}(r)=+\infty$, and $h_{\lambda
}(1)=-(1+\lambda)$. Hence there exists $r^*(\lambda) \in(0,1)$ such
that $h_{\lambda}$ is strictly positive on $ (0,r^*(\lambda)
 ]$, and strictly negative on $ [r^*(\lambda),1  ]$.
Thus, $\widehat{f}_{\lambda}$ is strictly increasing on $
[0,r^*(\lambda)  ]$, and strictly decreasing on $
[r^*(\lambda),1  ]$.

If $c>0$, we have
\[
\widehat{f}_{\lambda}(c \sqrt{2 \lambda})\underset{\lambda \rightarrow0}
{=}1-\bigl(c+c^{-1}\bigr)\sqrt{2 \lambda}+o(\sqrt{\lambda}).
\]
Let $\varepsilon>0$. Applying the last relation to $c=1-\varepsilon$,
$c=1$, and $c=1+\varepsilon$ shows that for $\lambda$ small enough,
$\widehat{f}_{\lambda}(\sqrt{2 \lambda}) > \widehat{f}_{\lambda
}((1-\varepsilon) \sqrt{2 \lambda})$ and $\widehat{f}_{\lambda}(\sqrt
{2 \lambda}) > \widehat{f}_{\lambda}((1+\varepsilon) \sqrt{2 \lambda
})$. Thus, for $\lambda$ small enough, $r^*(\lambda) \in
[(1-\varepsilon) \sqrt{2\lambda},(1+\varepsilon) \sqrt{2\lambda}
]$. We deduce that $  r^*(\lambda) \underset{\lambda
\rightarrow0}{\sim}\sqrt{2 \lambda}$.
\end{pf}
%
We can now prove our main result.

\begin{theorem} \label{div}
In $\Gamma$, $(v_{\lambda})$ does not converge when $\lambda$ goes
to $0$.
\end{theorem}

\begin{pf}
We are going to show that the sequence $(v_{\lambda}(1^{++}))$ does
not converge when $\lambda$ goes to 0.

Set $\lambda_m=2^{-4m-1}$ and $\mu_m=2^{-4m-3}$. Hence $\sqrt
{2\lambda_m}=2^{-2m}$ and
$\sqrt{2\mu_m}=2^{-2m-1}$. By Lemma \ref{maxf}, for $m$ large enough,
\[
\mathop{\argmax}_{n \in\mathbb{N}} f_{\lambda_m}=\mathop{\argmax}_{n \in2\mathbb
{N}}
f_{\lambda_m}= \{2m \}.
\]
Thus, for the discount factors $(\lambda_m)_{m \in\mathbb{N}}$, the
fact that Player~1 has a wider set of strategies than Player~2 does not
affect the outcome of the game. By Proposition~\ref{stratopt}, we have
\[
v_{\lambda_m}\bigl(1^{++}\bigr)= \bigl(1-f_{\lambda_m}(2m)
\bigr) \bigl(1-f_{\lambda_m}(2m)^2 \bigr)^{-1}=
\bigl(1+f_{\lambda_m}(2m) \bigr)^{-1}.
\]
By Lemma \ref{maxf}, $(f_{\lambda_m}(2m))$ converges to 1 when $m$
goes to infinity, thus $(v_{\lambda_m}(1^{++}))$ converges to
$ 1/2$.

By Lemma \ref{maxf} again, for $m$ large enough, we have
\[
\mathop{\argmax}_{n \in\mathbb{N}} f_{\mu_m}= \{2m+1 \}\quad \mbox{and}\quad
\mathop{\argmax}_{n \in2\mathbb{N}} f_{\mu_m} \subset \{2m,2m+2 \}.
\]
Contrary to the previous situation, Player~1 has an advantage over her
opponent: Player~2 cannot choose $2m+1$. In $\widetilde{\Gamma
}{}^{1^{++}}_{\mu_m}$, choosing $t(2m)$ or $t(2m+2)$ instead
of $t(2m+1)$ changes the dynamics of the state, and makes this
advantage substantial. Formally, we have
\[
v_{\mu_m}\bigl(1^{++}\bigr)=\min \biggl(\frac{1-f_{\mu_m}(2m)}{1-f_{\mu
_m}(2m)f_{\mu_m}(2m+1)},
\frac{1-f_{\mu_m}(2m+2)}{1-f_{\mu
_m}(2m+2)f_{\mu_m}(2m+1)} \biggr).
\]
By Lemma \ref{maxf}, we have
\begin{eqnarray*}
f_{\mu_m}(2m+1)&\underset{m \rightarrow+\infty} {=}& 1-2 \sqrt{2\mu
_m}+o(\sqrt{\mu_m}),
\\
f_{\mu_m}(2m)&\underset{m \rightarrow+\infty} {=}& 1-5/2 \sqrt{2\mu
_m}+o(\sqrt{\mu_m}),
\\
f_{\mu_m}(2m+2)&\underset{m \rightarrow+\infty} {=}& 1-5/2 \sqrt {2
\mu_m}+o(\sqrt{\mu_m}).
\end{eqnarray*}
Hence,
\[
\frac{1-f_{\mu_m}(2m)}{1-f_{\mu_m}(2m)f_{\mu_m}(2m+1)} \underset{m \rightarrow+\infty} {\sim} \frac{5/2\sqrt{2\mu_m}}{ (2+5/2 )\sqrt{2\mu_m}}=5/9
\]
and similarly
\[
\frac{1-f_{\mu_m}(2m+2)}{1-f_{\mu_m}(2m+2)f_{\mu_m}(2m+1)} \underset{m \rightarrow+\infty} {\sim} \frac{5/2 \sqrt{2\mu_m}}{ (2+5/2 )\sqrt{2\mu_m}}=5/9.
\]
The sequences $(v_{\lambda_m}(1^{++}))$ and $(v_{\mu_m}(1^{++}))$
converge to different limits, hence $(v_{\lambda})$ does not converge.
\end{pf}
%

\begin{remark}
More generally, for every initial state $p \in P \setminus \{
1^*,0^*  \}$, $(v_{\lambda}(p))$ does not converge. Consider,
for example, the case $p=0_n$, for some $n \in\mathbb{N}$. Let $N
\geq n$. Consider the following strategy $\sigma$ for Player~1 in
$\Gamma_{\lambda}^{0_n}$: play $C$ until \mbox{$p_m=0_N$}, then play $Q$,
and then play optimally in $\Gamma_{\lambda}^{1^{++}}$ if the state
$1^{++}$ is reached. For $\lambda$ small enough, the strategy $\sigma
$ guarantees $  v_{\lambda}(1^{++})-2/N$ in $\Gamma
_{\lambda}^{0_n}$. We deduce that $\lim_{\lambda\rightarrow0} \llvert  v_{\lambda}(0_n)-v_{\lambda}(1^{++})\rrvert  =0$. With\vspace*{1pt} a similar
argument, one can show that for all $(p,p') \in P^2 \setminus \{
1^*,0^*  \}$, $\lim_{\lambda\rightarrow0} \llvert  v_{\lambda
}(p)-v_{\lambda}(p')\rrvert  =0$, which gives the result.
\end{remark}

In this section, we have shown that Conjectures \ref{con1} and \ref{con2} are false, by presenting an example of a repeated game with public
signals and perfect observation of the actions where $(v_{\lambda})$
does not converge. In the following section, we construct a repeated
game belonging to the same class, where neither $(v_\lambda)$ nor
$(v_n)$ converge.

\section{From \texorpdfstring{$(v_{\lambda})$}{(v{lambda})} to $(v_n)$}\label{sec3} \label{vn}

\subsection{Motivation of the example}\label{sec3.1}
The idea of the construction of the game is based on the following
lemma, which can be deduced from the proof of \citet{sorin02b}, Theorem~C.8, page~177. We provide the proof for completeness.

\begin{lemma} \label{neyman}
Let $\Gamma$ be any repeated game with symmetric information, and $p
\in\Delta(K)$. Let $P \subset\Delta(K)$ be the set of states which
can be reached with positive probability under some strategy vector in
the game $\widetilde{\Gamma}{}^p$. Let $n_0,n \in\mathbb{N}^*$, and
for $m \in\mathbb{N}^*$ set $  w_m:=v_{1/m}$. Let $\llVert\cdot  \rrVert  $ denote the supremum over $P$. Then the following
inequality holds:
\[
\llVert v_n-w_n \rrVert \leq\frac{n_0}{n} \llVert
v_{n_0}-w_{n_0}\rrVert +\sum_{m=n_0}^{n-1}
\llVert w_m-w_{m+1} \rrVert.
\]
\end{lemma}

\begin{pf}
Let $m \geq1$ and $p \in P$. We have the following dynamic programming
principle [see \citet{sorin02b}, Properties C.13, page~181]:
%
%
\begin{eqnarray}
\label{prog1} \qquad v_m(p)&=&\max_{x \in\Delta(I)} \min
_{y \in\Delta(J)} \bigl\{m^{-1}g(p,x,y)+(m-1) m^{-1}
\mathbb{E}^p_{x,y}(v_{m-1}) \bigr\}
\\
\label{prog1d} &=& \min_{y \in\Delta(J)} \max_{x \in\Delta(I)}
\bigl\{m^{-1}g(p,x,y)+(m-1) m^{-1} \mathbb{E}^p_{x,y}(v_{m-1})
\bigr\}
\end{eqnarray}
and
%
%
\begin{eqnarray}
\label{prog2} \qquad w_m(p)&=& \max_{x \in\Delta(I)} \min
_{y \in\Delta(J)} \bigl\{m^{-1} g(p,x,y)+(m-1)m^{-1}
\mathbb{E}^p_{x,y}(w_m) \bigr\}
\\
&=& \label{prog2d} \min_{y \in\Delta(J)} \max_{x \in\Delta(I)}
\bigl\{m^{-1} g(p,x,y)+(m-1)m^{-1} \mathbb{E}^p_{x,y}(w_m)
\bigr\}.
\end{eqnarray}
%
Let $x \in\Delta(I)$ be optimal in (\ref{prog1}) and $y \in\Delta
(J)$ be optimal in (\ref{prog2d}). We have
\begin{eqnarray*}
v_m(p) &\leq& m^{-1} g(p,x,y)+(m-1) m^{-1}
\mathbb{E}^p_{x,y}(v_{m-1}),
\\
w_m(p) &\geq& m^{-1} g(p,x,y)+(m-1)m^{-1}
\mathbb{E}^p_{x,y}(w_{m}).
\end{eqnarray*}
The combination of these two inequalities gives
\[
v_{m}(p)-w_m(p) \leq(m-1)m^{-1}\llVert
v_{m-1}-w_m \rrVert.
\]
Taking $x' \in\Delta(I)$ optimal in (\ref{prog1d}) and $y' \in
\Delta(J)$ optimal in (\ref{prog2}) gives the symmetric inequality:
\[
w_m(p)-v_{m}(p) \leq(m-1)m^{-1}\llVert
v_{m-1}-w_m \rrVert.
\]
Hence,
\[
\llVert v_{m}-w_m\rrVert \leq(m-1)m^{-1}
\llVert v_{m-1}-w_m \rrVert
\]
and
\[
m \llVert v_{m}-w_m\rrVert \leq(m-1) \llVert
v_{m-1}-w_{m-1}\rrVert +(m-1)\llVert w_{m-1}-w_m
\rrVert.
\]
Let $n,n_0 \geq1$. Summing the last inequality from $n_0+1$ to $n$ yields
\[
n \llVert v_{n}-w_n\rrVert \leq n_0 \llVert
v_{n_0}-w_{n_0}\rrVert +\sum_{m=n_0}^{n-1}
m \llVert w_{m}-w_{m+1}\rrVert.
\]
Dividing by $n$ gives the desired result.
\end{pf}
%

We construct a family of repeated games with symmetric information
$(\Gamma(r))_{r \geq1}$, such that the value $ (v^r_{\lambda
} )$ of $\Gamma_{\lambda}(r)$ does not converge. Moreover, for
all $p \in\Delta(K)$, the derivative function of $v^r_{\lambda}(p)$
with respect to $\lambda$ is bounded by $C(r) \lambda^{-1}$ for
$\lambda$ sufficiently small, where $C(r)>0$ is independent of $p$ and
goes to $0$ as $r$ goes to infinity. Applying Lemma \ref{neyman}, we
are then able to show that for $r$ large enough, $(v^r_n)$ does not converge.

\subsection{Description of the game}\label{sec3.2}
Let $r \geq1$.
Consider the following repeated game with symmetric information $\Gamma
(r)$. The state space is
\[
K= \bigl\{1^{++},1^{T_1},1^{T_2},\ldots,1^{T_{2r-1}},1^{+},1^*,0^{++},0^{T_1},0^{T_2},
\ldots,0^{T_{r-1}},0^+,0^* \bigr\},
\]
the actions sets are $I=J= \{ C,Q  \}$, and the signal set
is $A= \{D,D'  \}$.

Payoffs are independent of actions, and are 1 in states belonging to
$ \{1^*,\break 1^{++}, 1^{T_1},\ldots,1^{T_{2r-1}},1^+ \}$, and 0
in states belonging to
$ \{0^*,0^{++},0^{T_1},\ldots,0^{T_{r-1}},\break 0^+  \}$.

Player~2 controls the states $1^{++},1^{T_1},\ldots,1^{T_{2r-1}},1^+$,
and Player~1 controls the states $0^{++},0^{T_1},\ldots,0^{T_{r-1}},0^+$.
Hence, $q$ can be seen as a map from $K \times \{C,Q  \}$
to $\Delta(K \times \{D,D' \})$. Lastly, states $1^*$ and
$0^*$ are absorbing states.
The next figure describes the transitions in the state $0^{T_{l}}$,
where $l \in \{0,1,\ldots,r-2  \}$, and in the states
$0^{T_{r-1}}$, $0^+$ and $0^*$ (by convention $0^{T_0}:=0^{++}$). To
simplify Figure~\ref{fig4}, the transitions in states $0^{T_0}=0^{++}$ and
$0^{T_{l+1}}$ are not represented.
%
\begin{figure}

\includegraphics{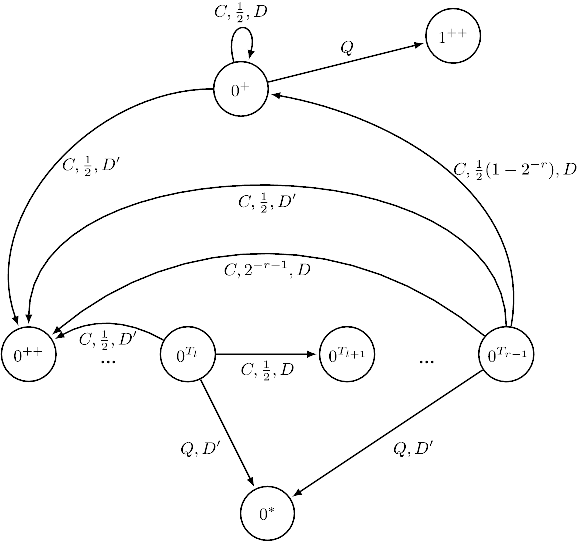}

\caption{Transitions in the states controlled by Player~1.}
\label{fig4}
%
%
%
%
%
%
\end{figure}

In the states controlled by Player~2, the transitions are analog: one
replaces $0$ by~$1$ and $r-1$ by $2r-1$ (with the convention $1^{T_0}:=1^{++}$).

\begin{remark}
The case $r=1$ corresponds to the example of Section~\ref{counter}.
\end{remark}

We now argue that the game $\widetilde{\Gamma}(r)$ presents strong
similarities with the game $\widetilde{\Gamma}$ introduced in
Section~\ref{counter}.
Let $m \in\mathbb{N}$ and $l \in \{0,1,\ldots,r-1  \}$.

Let $0_{mr+l}:=2^{-mr} \cdot0^{T_l}+(1-2^{-mr}) \cdot0^+$, and for $l
\in \{0,1,\ldots,2r-1  \}$ let
$1_{2mr+l}:=2^{-2mr} \cdot1^{T_l}+(1-2^{-2mr}) \cdot1^+$. Let
\[
P:=\cupp_{n \in\mathbb{N}} \{1_n,0_n \} \cup \bigl\{
1^*,0^* \bigr\}.
\]
The set $P$ is the set of all states which can be reached with positive
probability in $\widetilde{\Gamma}(r)^{1^{++}}$ (note that $P$ has
the same formal definition as the set $P$ in Section~\ref{red}).
Moreover, when a player plays $C$ in some state $p \in P$, the
transition is identical to the transition in our first example
$\widetilde{\Gamma}$ (see Section~\ref{red}). When Player~1 (resp.,~2) plays $Q$ at state\vspace*{2pt} $0_{mr+l}$ (resp., $1_{2mr+l}$), the state is
absorbed in state $0^*$ (resp., $1^*$) with probability $2^{-mr}$
(resp., $2^{-2mr}$). Thus the only difference with $\widetilde{\Gamma
}$ is that in $\widetilde{\Gamma}(r)$, Player~1 can only take
absorbing risks in the set $ \{2^{-mr}, m \in\mathbb{N}  \}
$, and Player~2 can only take absorbing risks in the set $ \{
2^{-2mr}, m \in\mathbb{N}  \}$. The next proposition is the
equivalent of Proposition \ref{eq2}, and its proof is identical.



\begin{proposition} \label{eql}
The game $\widetilde{\Gamma}(r)^{1^{++}}_{\lambda}$ has the same
value as the strategic-form game $G_{\lambda}(r)$, with action set $r
\mathbb{N}$ for Player~1, $2 r \mathbb{N}$ for Player~2, and payoff
\[
g^r_{\lambda}(a,b):=g_{\lambda}(a,b),
\]
where\vspace*{1pt} $g_{\lambda}$ is defined in (\ref{discpayoff}).
Moreover, optimal strategies in $G_{\lambda}(r)$ induce optimal
strategies in $\widetilde{\Gamma}(r)^{1^{++}}_{\lambda}$.
\end{proposition}

\subsection{Asymptotic study of \texorpdfstring{$G_{\lambda}(r)$}{G{lambda}(r)}}\label{sec3.3} Let $r \geq2$. For
$m \geq1$, let
$\lambda_m:=2^{-4mr-1}$ and $\mu_m:=2^{-4mr-2r-1}$.
Hence, $\sqrt{2 \lambda_m}=2^{-2m r}$ and
$\sqrt{2 \mu_m}=2^{-(2m+1)r}$.
Proceeding exactly the same way as in Section~\ref{limit}, we get the
following proposition.

\begin{proposition} \label{vl}
\[
\lim_{m \rightarrow+\infty} v_{\lambda_m}^r
\bigl(1^{++}\bigr)=1/2\quad\mbox {and}\quad \lim_{m \rightarrow+\infty}
v_{\mu_m}^r\bigl(1^{++}\bigr)=\frac
{2^r+2^{-r}}{2^r+2^{-r}+2}.
\]
\end{proposition}

\begin{pf*}{Sketch of proof}
For $m$ large enough, we have by Lemma \ref{maxf}
\[
\mathop{\argmax}_{n \in r\mathbb{N}} f_{\lambda_m}=\mathop{\argmax}_{n \in2r\mathbb
{N}}
f_{\lambda_m}= \{2mr \}.
\]

We deduce (see the proof of Theorem \ref{div}) that $(v^r_{\lambda
_m})(1^{++})$ converges to $1/2$. By Lemma \ref{maxf} once again, for
$m$ large enough, we have
\[
\mathop{\argmax}_{n \in r\mathbb{N}} f_{\mu_m} = \bigl\{(2m+1)r \bigr\}
\]
and
\[
\mathop{\argmax}_{n \in2r\mathbb{N}} f_{\mu_m} \subset \bigl
\{2mr,2(m+1)r \bigr\}.
\]
We deduce (see the proof of Theorem \ref{div}) that $v^r_{\mu
_m}(1^{++})$ converges to $(2^r+2^{-r})/(2^r+2^{-r}+2)$.
\end{pf*}
We now use Lemma \ref{neyman} to show that the value $(v^r_n)$ of the
game $\Gamma(r)_n$ does not converge. We adopt the following notation:
if $(\lambda,p) \in(0,1] \times\Delta(K)$, we call $(v^r_{\lambda
})'(p)$ the derivative\vspace*{1pt} of $v^r_{\lambda}(p)$ with respect to $\lambda
$, evaluated in $\lambda$. Analogously, $f_{\lambda}'(p)$ is the
derivative of $f_{\lambda}$ with respect to $\lambda$, evaluated in
$\lambda$.


The following lemma gives a majorization of $(v^r_{\lambda})'(p)$ on
certain subintervals.

\begin{lemma} \label{boundedvar}
There\vspace*{2pt} exists $m_0 \geq1$ such that for all $m \geq m_0$, there exists
$\mu^0_m \in [\mu_m,2^{r/2-1} \mu_m  ]$ such that for
all $p \in P$ and for all $\mu\in [\mu_m,2^{r/2-1}\mu_m
]\setminus \{\mu^0_m  \}$, $v^r_{\mu}(p)$ is
differentiable at $\mu$ and
%
%
\begin{equation}
\label{diff1} \bigl\llvert \bigl(v^r_{\mu}
\bigr)'(p)\bigr\rrvert \leq2^{-r/2} \mu^{-1},
\end{equation}
and for all $\lambda_m \leq\lambda\leq2^{r-1}\lambda_m$,
$v^r_{\lambda}(p)$ is differentiable at $\lambda$ and
%
%
\begin{equation}
\label{diff2} \bigl\llvert \bigl(v^r_{\lambda}
\bigr)'(p)\bigr\rrvert \leq2^{r/2+2} \lambda^{-1/2}.
\end{equation}
\end{lemma}

\begin{pf}
We start by proving inequality (\ref{diff1}). The proof proceeds in
three steps.

%
\begin{step}[{[\textit{Computation of $v^r_{\mu}(1^{++})$}]}]\label{st1}
By Lemma \ref{maxf}, there exists $m_1 \in\mathbb{N}^*$ such that
for all $m \geq m_1$ and $\mu\in [\mu_m,2^{r/2-1}\mu_m  ]$,
%
%
\[
\mathop{\argmax}_{n \in r\mathbb{N}} f_{\mu}= \{2mr+r \}.
\]
Hence, for such $m$ and $\mu$, $a(\mu):=2mr+r$ is an optimal strategy
for Player~1 in~$G_{\mu}(r)$.

By Lemma \ref{maxf} once again, there exists $m_2 \in\mathbb{N}^*$
such that for all $m \geq m_2$, there exists $\mu^0_m \in [\mu
_m,2^{r/2-1} \mu_m  ]$ such that for all $\mu\in(\mu
^0_m,2^{r/2-1} \mu_{m}]$
\[
\mathop{\argmax}_{n \in2r\mathbb{N}} f_{\mu}= \{2mr \},
\]
and for all $\mu\in[\mu_{m},\mu^0_m)$
\[
\mathop{\argmax}_{n \in2r\mathbb{N}} f_{\mu}= \{2mr+2r \}.
\]

Hence, for all $m \geq m_2$ and $\mu\in(\mu^0_m,2^{r/2-1} \mu
_{m}]$, the integer $b_1(\mu):=2mr$ is an optimal strategy for Player~2 in $G_{\mu}(r)$, and for all $\mu\in[\mu_{m},\mu^0_m)$, the
integer $b_2(\mu):=2mr+2r$ is an optimal\vspace*{1pt} strategy for Player~2 in
$G_{\mu}(r)$.
Thus, for all $m \geq \max(m_1,m_2):=m_3$ and $\mu\in (\mu
^0_m,2^{r/2-1}\mu_m  ]$, we have
%
%
\begin{eqnarray}\label{v1r}
v^r_{\mu}\bigl(1^{++}\bigr) &=&
g_{\lambda}\bigl(a(\mu),b_1(\mu)\bigr)
\nonumber\\[-8pt]\\[-8pt]\nonumber
&=& \bigl[1-f_{\lambda}
\bigl(b_1(\mu)\bigr) \bigr] \bigl[1-f_{\lambda}\bigl(a(\mu)
\bigr)f_{\lambda}\bigl(b_1(\mu)\bigr) \bigr]^{-1},
\end{eqnarray}
and for $\mu\in [\mu_m,\mu^0_m  )$
%
%
\begin{eqnarray}\label{v2r}
v^r_{\mu}\bigl(1^{++}\bigr) &=&
g_{\lambda}\bigl(a(\mu),b_2(\mu)\bigr)
\nonumber\\[-8pt]\\[-8pt]\nonumber
&=& \bigl[1-f_{\lambda}
\bigl(b_2(\mu)\bigr) \bigr] \bigl[1-f_{\lambda}\bigl(a(\mu)
\bigr)f_{\lambda}\bigl(b_2(\mu)\bigr) \bigr]^{-1}.
\end{eqnarray}
\end{step}

%
\begin{step}[{[\textit{Asymptotic expansion of $(v^r_{\mu})'(1^{++})$ as $\mu
\rightarrow0$}]}]\label{st2}
Fix $m \geq m_3$ and $\mu\in (\mu^0_m,2^{r/2-1}\mu_m  ]$.
Define $C_1\dvtx (0,1] \rightarrow\mathbb{R}_{+}^*$ and $C_2\dvtx (0,1]
\rightarrow\mathbb{R}_{+}^*$ by
%
%
\begin{equation}
C_1(\mu):=2^{-2mr-r}(2\mu)^{-1/2}= (
\mu_{m} )^{1/2} \mu^{-1/2}
\end{equation}
and
%
%
\begin{equation}
C_2(\mu):=2^{-2mr}(2\mu)^{-1/2}=2^r (
\mu_{m} )^{1/2} \mu ^{-1/2} \cdot
\end{equation}
Note that $ 2^{-a(\mu)}=C_1(\mu) \sqrt{2\mu}$ and
$ 2^{-b_1(\mu)}=C_2(\mu) \sqrt{2\mu}$. Moreover,\break  $1
\leq \llVert  C_1 \rrVert  _{\infty} \leq2^r$ and $2^r \leq\llVert
C_2 \rrVert  _{\infty} \leq2^{2r}$. The last two inequalities show
that the functions $C_1$ and $C_2$ are bounded and bounded away from
$0$, which will be useful in the following asymptotic expansions.

Since the functions $\lambda\rightarrow a(\lambda)$ and $\lambda
\rightarrow b_1(\lambda)$ are constant on $ (\mu
^0_m,2^{r/2-1}\mu_m  ]$, $\lambda\rightarrow f_{\lambda
}(a(\lambda))$ and $\lambda\rightarrow f_{\lambda}(b_1(\lambda))$
are differentiable at $\mu$, and we have
\begin{eqnarray*}
f_{\mu}'\bigl(a(\mu)\bigr)&=& \bigl(1-2^{-a(\mu)}
\bigr) \bigl(1+2^{a(\mu
)+1}\mu(1-\mu)^{-a(\mu)}-\mu \bigr)^{-2}
\\
&&{}\times \bigl[-2\mu \bigl(1+2^{a(\mu)+1}\mu(1-\mu)^{-a(\mu)}-\mu \bigr)
\\
&&\hspace*{15pt}{} -\bigl(1-\mu^2\bigr) \bigl(2^{a(\mu)+1} \bigl[(1-
\mu)^{-a(\mu)}-a(\mu ) \mu(1-\mu)^{a(\mu)-1} \bigr]-1 \bigr) \bigr].
\end{eqnarray*}
%
We deduce that
%
%
\begin{equation}
\label{df} f_{\mu}'\bigl(a(\mu)\bigr)\underset{\mu
\rightarrow0} {=}-2C_1(\mu)^{-1}(2\mu )^{-1/2}+o
\bigl(\mu^{-1/2} \bigr),
\end{equation}
and likewise
%
%
\begin{equation}
\label{df2} f_{\mu}'\bigl(b_1(\mu)\bigr)
\underset{\mu\rightarrow0} {=}-2C_2(\mu )^{-1}(2
\mu)^{-1/2}+o \bigl(\mu^{-1/2} \bigr).
\end{equation}
%
The same computation as in Lemma \ref{maxf} gives
\begin{eqnarray*}
f_{\mu}\bigl(a(\mu)\bigr) &\underset{\mu\rightarrow0} {=}&1-
\bigl(C_1(\mu)+C_1(\mu )^{-1}\bigr) (2
\mu)^{1/2}+o \bigl(\mu^{1/2} \bigr),
\\
f_{\mu}\bigl(b_1(\mu)\bigr)&\underset{\mu\rightarrow0}
{=}&1-\bigl(C_2(\mu)+C_2(\mu )^{-1}\bigr) (2
\mu)^{1/2}+o \bigl(\mu^{1/2} \bigr).
\end{eqnarray*}
Now\vspace*{1pt} we differentiate $v^r_{\mu}(1^{++})$ in (\ref{v1r}) (we omit the
dependence of $a$ and $b_1$ on~$\mu$). The derivative $(v^r_{\mu
})'(1^{++})=g_{\mu}'(a,b_1)$ is equal to
\begin{eqnarray*}
&&\frac{-f'_{\mu}(b_1)(1-f_{\mu}(a)f_{\mu}(b_1))+(1-f_{\mu
}(b_1))(f_{\mu}'(a)f_{\mu}(b_1)+f_{\mu}(a)f_{\mu}'(b_1))}{(1-f_{\mu
}(a)f_{\mu}(b_1))^2}
\\
&&\qquad = \frac{f_{\mu}'(a)f_{\mu}(b_1)(1-f_{\mu}(b_1))-f_{\mu
}'(b_1)(1-f_{\mu}(a))}{(1-f_{\mu}(a)f_{\mu}(b_1))^2}.
\end{eqnarray*}
When $\mu$ goes to 0, the numerator of this expression is (we omit the
dependence on~$\mu$)
\[
-2C_1^{-1}\bigl(C_2+C_2^{-1}
\bigr)+2C_2^{-1}\bigl(C_1+C_1^{-1}
\bigr)+o(1)=2\bigl(C_2^{-1}C_1-C_1^{-1}C_2
\bigr)+o(1).
\]
Hence,
\begin{eqnarray*}
\bigl(v^r_{\mu}\bigr)'\bigl(1^{++}
\bigr)&\underset{\mu\rightarrow0} {=}&\frac
{(C_2^{-1}C_1-C_1^{-1}C_2)}{(C_1+C_1^{-1}+C_2+C_2^{-1})^2}\mu ^{-1}+o
\bigl(\mu^{-1}\bigr).
\end{eqnarray*}
Since $\mu\leq2^{r/2-1} \mu_m$, we have
\[
\biggl\llvert \frac
{(C_2^{-1}C_1-C_1^{-1}C_2)}{(C_1+C_1^{-1}+C_2+C_2^{-1})^2}\biggr\rrvert \leq C_1^{-1}C_2^{-1}=2^{-r}
\frac{\mu}{\mu_{m}} \leq2^{-r/2-1}.
\]
The last two relations show that for $m$ large enough and $\mu^0_m <
\mu\leq2^{r/2-1} \mu_m$
%
%
\begin{equation}
\label{incaseg} \bigl\llvert \bigl(v^r_{\mu}\bigr)
\bigl(1^{++}\bigr)\bigr\rrvert \leq2^{-r/2-1} \mu^{-1}.
\end{equation}
Equation (\ref{v2r}) and similar computations show that the last
inequality is also true for $m$ large enough and $\mu\in [\mu
_m, \mu^0_m  )$.
\end{step}
%

\begin{step}[{[\textit{Computation of} $(v^r_{\mu})'(p)$ \textit{for} $p \in P$ \textit{and proof of
inequality} (\ref{diff1})]}]\label{st3}
For $a \in\mathbb{N}$ and $n \leq a$, let $T^n_{a}$ be the random
time needed by Player~1 to go from state $0_n$ to state $0_a$, when she
plays strategy $s(a)$. Let $m \geq m_3$ and $\mu\in (\mu
^0_m,2^{r/2-1}\mu_m  ]$.

If $n < a(\mu)$, then
%
%
\begin{equation}
\label{calcul0} v^r_{\mu}(0_n)=
\bigl(1-2^{-a(\mu)}\bigr) \mathbb{E} \bigl((1-\mu)^{T^n_{a(\mu)}+1}
\bigr)v^r_{\mu}\bigl(1^{++}\bigr).
\end{equation}
If $n \geq a(\mu)$, then
%
%
\begin{equation}
v^r_{\mu}(0_n)=\bigl(1-2^{-n}\bigr)
(1-\mu) v^r_{\mu}\bigl(1^{++}\bigr).
\end{equation}
Thus, for all $n \in\mathbb{N}$
\begin{eqnarray}\label{incase0}
\nonumber
\bigl\llvert \bigl(v^r_{\mu}
\bigr)'(0_n)\bigr\rrvert &\leq& \mathbb{E}
\bigl(T^n_{a(\mu
)}+1\bigr)+\bigl\llvert \bigl(v^r_{\mu}
\bigr)'\bigl(1^{++}\bigr)\bigr\rrvert
\nonumber\\[-8pt]\\[-8pt]\nonumber
&\leq& \mathbb{E}(T_{a(\mu)})+\bigl\llvert
\bigl(v^r_{\mu}\bigr)'\bigl(1^{++}
\bigr)\bigr\rrvert.
\end{eqnarray}
Similar arguments lead to
%
%
\begin{equation}
\label{incase1} \bigl\llvert v^r_{\mu}(1_n)
\bigr\rrvert \leq\mathbb{E}(T_{b_1(\mu)})+\bigl\llvert \bigl(v^r_{\mu}
\bigr)'\bigl(0^{++}\bigr)\bigr\rrvert.
\end{equation}
By (\ref{step3}), $\mathbb{E} (T_{a(\mu)} )=o(\mu^{-1})$
and $\mathbb{E} (T_{b_1(\mu)} )=o(\mu^{-1})$ when $\mu$
goes to $0$.

Combining inequalities (\ref{incaseg}), (\ref{incase0}) and (\ref
{incase1}), we get inequality (\ref{diff1}) for $m$ large enough and
$\mu\in (\mu^0_m,2^{r/2-1}\mu_m  ]$.
For $m$ large enough and $\mu\in (\mu_m,\mu^0_m  ]$, the
computations are similar.
\end{step}

Let us now prove inequality (\ref{diff2}). By Lemma \ref{maxf}, there
exists $m_4 \in\mathbb{N}^*$ such that for all $\lambda\in
[\lambda_m,2^{r-1}\lambda_m  ]$, $a(\lambda):=2mr$ is an
optimal strategy for both players in $G_{\lambda}(r)$, and
\[
\bigl(v^r_{\lambda}\bigr)'\bigl(1^{++}
\bigr)=-\frac{f_{\lambda}(a(\lambda
))'}{(1+f_{\lambda}(a(\lambda)))^2}.
\]
For $m \geq m_4$ and $\lambda\in [\lambda_m,2^{r-1}\lambda_m
 ]$, let $  C(\lambda):=\sqrt{\frac{\lambda
_{m}}{\lambda}}$. Then as in (\ref{df})
\[
f_{\lambda}'\bigl(a(\lambda)\bigr)\underset{\lambda
\rightarrow 0} {=}-2C(\lambda)^{-1}(2\lambda)^{-1/2}+o \bigl(
\lambda^{-1/2} \bigr).
\]
Since $  C(\lambda)^{-1} \leq2^{(r-1)/2}$ and $f_{\lambda
}(a(\lambda))$ goes to 1 when $\lambda$ goes to 0, we get inequality
(\ref{diff2}) for $(v^r_{\lambda})'(1^{++})$ and it extends to
$(v^r_{\lambda})'(p)$ with the method used in Step~\ref{st3}.
\end{pf}

%
\begin{theorem}
There exists $r_0 \in\mathbb{N}^*$ such that for all $r \geq r_0$,
$(v^r_n)$ and $(v^r_{\lambda})$ do not converge.
\end{theorem}

\begin{pf}
Let $r \geq2$. Recall that from Proposition \ref{vl},
\[
\lim_{m \rightarrow+ \infty} v^r_{\lambda_m}
\bigl(1^{++}\bigr)=1/2\quad\mbox {and}\quad\lim_{m \rightarrow+\infty}
v^r_{\mu_m}\bigl(1^{++}\bigr) =
\frac
{2^r+2^{-r}}{2^r+2^{-r}+2}:=w(r).
\]
In particular, $(v^r_{\lambda})$ does not converge. We are going to
show that for $r$ big enough, $(v^r_n(1^{++}))$ does not converge.

Let $m_0$ be as in Lemma \ref{boundedvar} and let $m \geq m_0$. Let
$n(m):=\mu_m^{-1}=2^{4mr+2r+1}$ and $n_0(m):=2^{- \lfloor r/2
 \rfloor+1}n(m)$. We now compare $v^r_{n(m)}$ and $v^r_{\mu
_m}$. Using Lemma \ref{neyman}, we get
\begin{eqnarray*}
\bigl\llVert v^r_{n(m)}-v^r_{\mu_m}
\bigr\rrVert & \leq& \frac{n_0(m)}{n(m)} \bigl\llVert v^r_{n_0(m)}-v^r_{2^{
\lfloor r/2  \rfloor-1}\mu_m}
\bigr\rrVert
+ \sum_{m'=n_0(m)}^{n(m)-1} \bigl\llVert
v^r_{1/m'}-v^r_{1/(m'+1)}\bigr\rrVert.
\end{eqnarray*}
Since the payoff function is bounded by 1, we have
\[
\frac{n_0(m)}{n(m)} \bigl\llVert v^r_{n_0(m)}-v^r_{2^{ \lfloor r/2
 \rfloor-1}\mu_m}
\bigr\rrVert \leq2^{- \lfloor r/2  \rfloor+1}.
\]

By inequality (\ref{diff1}) in Lemma \ref{boundedvar} and the mean
value theorem, we have
\begin{eqnarray*}
\sum_{m'=n_0(m)}^{n(m)-1} \bigl\llVert
v^r_{1/m'}-v^r_{1/(m'+1)}\bigr\rrVert &
\leq& 2^{-r/2+1} \sum_{m'=n_0(m)}^{n(m)-1} \int
_{1/(m'+1)}^{1/m'} \frac{1}{x}\,dx
\\
&=& 2^{-r/2+1} \int_{(n(m))^{-1}}^{(n_0(m))^{-1}}
\frac{1}{x}\,dx
\\
&=& 2^{-r/2+1} \bigl( \lfloor r/2 \rfloor-1 \bigr).
\end{eqnarray*}
Letting $m$ going to infinity, we deduce that
%
\[
\limsup_{n \rightarrow+\infty} v^r_n
\bigl(1^{++}\bigr) \geq w(r)-2^{-
\lfloor r/2  \rfloor+1}-2^{-r/2+1} \bigl(
\lfloor r/2 \rfloor-1 \bigr).
\]
Note that $\lim_{r \rightarrow+\infty} w(r)=1$ and that the term on
the right goes to 0 when $r$ goes to infinity.

Lemma \ref{neyman} for $n(m):=\lambda_m^{-1}$ and
$n_0(m):=2^{-\lfloor r/2 \rfloor+1}n(m)$ gives also an inequality of
the form
%
\[
\liminf_{n \rightarrow+\infty} v^r_n
\bigl(1^{++}\bigr) \leq1/2+t(r),
\]
where $\lim_{r \rightarrow+\infty} t(r)=0$. Hence, for $r$ large
enough, $(v^r_n(1^{++}))$ does not converge, thus $(v^r_n)$ does not converge.
\end{pf}
We have proved in this section that in a repeated game with symmetric
information, the value of the $n$-stage repeated game might not
converge. To do so, we have exploited the very flexible structure of
the first example of Section~\ref{sec2}. Indeed, we have managed to slow down
the oscillations
of $v_{\lambda}$, without changing much the dynamics of the game.

In the next section, we again take advantage of the flexibility of the
game, to provide other examples of
repeated games without an asymptotic value.
\section{Extension to other classes of repeated games}\label{sec4}
\subsection{State-blind repeated games}\label{sec4.1} \label{stateblind}
Consider the following state-blind repeated game $\Gamma$, with state
space $K= \{1^*,1^{++},1^T,1^{+},0^*,0^{++},0^+  \}$, action
sets $I= \{T,B,Q  \}$ for Player~1 and $J= \{L,R,Q
 \}$ for Player~2. The states $0^*$ and $1^*$ are absorbing
states. The payoff is 1 in states $1^{++}$, $1^{T}$, $1^+$ and $1^*$,
and 0 in states $0^{++}$, $0^+$ and $0^*$. The transitions are
described in Tables \ref{tab1}--\ref{tab5}.

Recall that in this model, the players do not observe any signal about
the state, and only observe past actions.

%
\begin{table}[b]
\tabcolsep=0pt
\tablewidth=200pt
\caption{State $1^{++}$}\label{tab1}
\begin{tabular*}{\tablewidth}{@{\extracolsep{\fill}}@{}lccc@{}}
\hline
& $\bolds{L}$ & $\bolds{R}$ & $\bolds{Q}$\\
\hline
$T$ & $1^{++}$ & $1^T$ & $1^*$ \\
$B$ & $1^{T}$ & $1^{++}$ & $1^*$ \\
$Q$ & $0^*$ & $0^*$ & $0^*$ \\
\hline
\end{tabular*}
\end{table}

\begin{table}[b]
\tabcolsep=0pt
\tablewidth=200pt
\caption{State $1^{T}$}\label{tab2}
\begin{tabular*}{\tablewidth}{@{\extracolsep{\fill}}@{}lccc@{}}
\hline
& $\bolds{L}$ & $\bolds{R}$ & $\bolds{Q}$\\
\hline
$T$ & $1^{++}$ & $3/4 \cdot1^+ + 1/4 \cdot1^{++}$ & $1^*$ \\
$B$ & $3/4 \cdot1^+ + 1/4 \cdot1^{++}$ & $1^{++}$ & $1^*$ \\
$Q$ & $0^*$ & $0^*$ & $0^*$ \\
\hline
\end{tabular*}
\end{table}

\begin{table}
\tabcolsep=0pt
\tablewidth=200pt
\caption{State $1^{+}$}\label{tab3}
\begin{tabular*}{\tablewidth}{@{\extracolsep{\fill}}@{}lccc@{}}
\hline
& $\bolds{L}$ & $\bolds{R}$ & $\bolds{Q}$\\
\hline
$T$ & $1^{++}$ & $1^+$ & $0^{++}$ \\
$B$ & $1^+$ & $1^{++}$ & $0^{++}$ \\
$Q$ & $0^*$ & $0^*$ & $0^*$ \\
\hline
\end{tabular*}
\end{table}

\begin{table}
\tabcolsep=0pt
\tablewidth=200pt
\caption{State $0^{++}$}\label{tab4}
\begin{tabular*}{\tablewidth}{@{\extracolsep{\fill}}@{}lccc@{}}
\hline
& $\bolds{L}$ & $\bolds{R}$ & $\bolds{Q}$\\
\hline
$T$ & $1/2 \cdot0^{++}+1/2 \cdot0^{+}$ &$0^{++}$ & $1^*$ \\
$B$ & $0^{++}$ & $1/2 \cdot0^{++}+1/2 \cdot0^{+}$ & $1^*$ \\
$Q$ & $0^*$ & $0^*$ & $1^*$ \\
\hline
\end{tabular*}
\end{table}

\begin{table}
\tabcolsep=0pt
\tablewidth=200pt
\caption{State $0^{+}$}\label{tab5}
\begin{tabular*}{\tablewidth}{@{\extracolsep{\fill}}@{}lccc@{}}
\hline
& $\bolds{L}$ & $\bolds{R}$ & $\bolds{Q}$\\
\hline
$T$ & $0^+$ &$0^{++}$ & $1^*$ \\
$B$ & $0^{++}$ & $0^+$ & $1^*$ \\
$Q$ & $1^{++}$ & $1^{++}$ & $1^*$ \\
\hline
\end{tabular*}
\end{table}

The idea of this example is to artificially recreate the dynamics of
the example of Section~\ref{counter}, replacing signals by the mixed
actions of one player.


Formally, let $a^*(\lambda)$ [resp., $b^*(\lambda)$] be defined as in
Proposition \ref{stratopt}.

Let $\sigma^* \in\Sigma$ be the following strategy for Player~1 in
$\Gamma^{1^{++}}_{\lambda}$: play $1/2 \cdot T +1/2 \cdot B$ if
$\mathbb{P}(k_m=0^{+}\mid k_m \notin \{0^*,1^* \}) \leq
1-2^{-a^*(\lambda)}$, otherwise play $Q$.

Let $\tau^* \in\mathcal{T}$ be the following strategy for Player~2:
play $1/2 \cdot L +1/2 \cdot R$ if $\mathbb{P}(k_m=1^{+}\mid k_m \notin
 \{0^*,1^* \}) \leq1-2^{-b^*(\lambda)}$, otherwise play $Q$.

Proceeding as in Section~\ref{counter}, one can show that $\sigma^*$
and $\tau^*$ are optimal strategies, respectively, for Player~1 and 2 in
$\Gamma^{1^{++}}_{\lambda}$.
Moreover, the probability measure on the histories of the game induced
by these strategies is the same as the
probability measure induced by $s(a^*(\lambda))$ and $t(b^*(\lambda
))$ in the example of Section~\ref{counter}. In particular, the two
examples have the same discounted value, thus $(v_{\lambda}(1^{++}))$
does not converge.


\subsection{Repeated games with one informed player}\label{sec4.2}
We now investigate a repeated game with perfect observation of the
actions, where Player~2 is fully informed about the state, while Player~1 has no information about it. As usual, both players observe past actions.

The state space is $K= \{1^*,1,0^*,0^{++},0^+  \}$, action
sets are $I= \{T,B,Q  \}$ for Player~1 and $J= \{L,R
 \}$ for Player~2. The states $0^*$ and $1^*$ are absorbing
states. The payoff is 1 in states $1$ and $1^*$, and 0 in states
$0^{++}$, $0^+$ and $0^*$. The transitions are described in the
following Tables \ref{tab6}--\ref{tab8}.

%
\begin{table}
\tabcolsep=0pt
\tablewidth=180pt
\caption{State $1$}\label{tab6}
\begin{tabular*}{\tablewidth}{@{\extracolsep{\fill}}@{}lcc@{}}
\hline
& $\bolds{L}$ & $\bolds{R}$\\
\hline
$T$ & $1$ & $0^{++}$ \\
$B$ & $0^{++}$ & $1^*$ \\
$Q$ & $0^*$ & $0^*$ \\
\hline
\end{tabular*}
\end{table}

\begin{table}
\tabcolsep=0pt
\tablewidth=180pt
\caption{State $0^{++}$}\label{tab7}
\begin{tabular*}{\tablewidth}{@{\extracolsep{\fill}}@{}lcc@{}}
\hline
& $\bolds{L}$ & $\bolds{R}$\\
\hline
$T$ & $1/2 \cdot0^{++}+1/2 \cdot0^{+}$ &$0^{++}$ \\
$B$ & $0^{++}$ & $1/2 \cdot0^{++}+1/2 \cdot0^{+}$ \\
$Q$ & $0^*$ & $0^*$ \\
\hline
\end{tabular*}
\end{table}

\begin{table}
\tabcolsep=0pt
\tablewidth=180pt
\caption{State $0^{+}$}\label{tab8}
\begin{tabular*}{\tablewidth}{@{\extracolsep{\fill}}@{}lcc@{}}
\hline
& $\bolds{L}$ & $\bolds{R}$\\
\hline
$T$ & $0^+$ &$0^{++}$ \\
$B$ & $0^{++}$ & $0^+$ \\
$Q$ & $1$ & $1$ \\
\hline
\end{tabular*}
\end{table}

Compared to the game of Section~\ref{stateblind}, states $1^{++}$, $
1^{T}$ and $1^{+}$ have been replaced by a single state $1$, which is
similar to the state $\omega^+$ in \citet{vigeral13}. The other states
have not been changed.

Let $\sigma^*$ be the following strategy for Player~1:
in state 1, play $(1-\sqrt{\lambda}) \cdot T + \sqrt{\lambda} \cdot
B$, and when the belief is in
$\Delta ( \{0^{++},0^+  \}  )$, play the same
strategy as in the preceding example.

Let $\tau^*$ be the following strategy for Player~2:
in state 1, play $(1-\sqrt{\lambda}) \cdot L + \sqrt{\lambda} \cdot
R$, and in states $0^{++}$ and $0^+$, play $(1/2 \cdot L + 1/2 \cdot R)$.
Proceeding the same way as in Section~\ref{counter}, one can show that
$\sigma^*$ and $\tau^*$ are asymptotically optimal strategies in
$\Gamma^1_{\lambda}$, and that $(v_{\lambda}(1))$ does not converge.


\subsection{Stochastic games with compact action sets}\label{sec4.3}
We now study a repeated game with perfect observation (states and
actions are known by both players) but where $I$ and $J$ are compact.
As mentioned in the \hyperref[sec1int]{Introduction}, this example does not relate to the
two conjectures, because $I$ and $J$ are not finite. But it yields a
simpler alternative counterexample to \citet{vigeral13}. This example
is similar to the example of Section \ref{counter} in terms of dynamics.

The state space is $K= \{1^*,1,0^*,0  \}$, and actions sets
are $I=[0,1]$ and
$J= \{0 \} \cup\bigcup_{m \in\mathbb{N}} 4^{-m}$. The
transition $q$ is defined by
\begin{eqnarray*}
q(1,x,y)&:=&(1-y) \cdot1+\bigl(y-y^2\bigr) \cdot0+y^2
\cdot1^*,
\\
q(0,x,y)&:=&(1-x) \cdot0+\bigl(x-x^2\bigr) \cdot1+x^2
\cdot0^*.
\end{eqnarray*}
Hence, Player~1 controls state 0 and Player~2 controls state 1.

Let $\lambda\in(0,1]$. A pure stationary strategy in $\Gamma
_{\lambda}$ for Player~1 (resp., 2) can be seen as an element of $I$
(resp., $J$).

\begin{remark}
The real number $x \in I$ corresponds to the absorbing risk $2^{-a}$ in
the example of Section~\ref{counter}. Indeed, when Player~1 plays $x$
in state 0, she waits on average $  x^{-1}$ stages before
switching to state 1, and the probability of absorbing in~$0^*$ before
reaching state 1 is approximately $x$. Recall that in the example of
Section~\ref{counter}, when Player~1 plays $a \in\mathbb{N}$, she
waits on average $2^a$ stages before quitting, and when she plays $Q$
the game is absorbed in state $0^*$ with probability $2^{-a}$. It is
the same for Player~2. As in our first example, Player~2 cannot take
any absorbing risk: only $y=4^{-m}$ for some $m \in\mathbb{N}$, or
$y=0$. But Player~1 can take any absorbing risk in $[0,1]$. That is why
we expect $(v_{\lambda})$ to oscillate, just as in the first example.
\end{remark}

The payoff in $\Gamma^1_{\lambda}$ given by a pair of strategies
$(x,y) \in I \times J$ is
\[
\gamma^1_{\lambda}(x,y)=\frac{(1-(1-\lambda)(1-y^2))(1-(1-\lambda
)(1-x))}{(1-(1-\lambda)(1-x y))(1-(1-\lambda)(1-x)(1-y))}.
\]
For any $x \in[0,1]$ (resp., $y \in[0,1]$) $\gamma^1_{\lambda}(x,\cdot)$
[resp., $\gamma^1_{\lambda}(\cdot ,y)$] is convex (resp., concave) and
reaches its minimum (resp., its maximum) at $y^*$ (resp., $x^*$) such that
\[
x^*=y^*=(\sqrt{\lambda}-\lambda) (1-\lambda)^{-1}.
\]
For $m \geq1$, we define $\lambda_m:=2^{-2m}$ and $\mu
_m:=2^{-2m-1}$. Then for $m$ large enough, $x_m=y_m=\sqrt{\lambda_m}$
are asymptotically optimal strategies in $\Gamma^1_{\lambda_m}$. Thus,
\[
\lim_{m \rightarrow+\infty} v_{\lambda_m}(1)=\lim_{m \rightarrow
+\infty}
\gamma^1_{\lambda_m}(x_m,y_m)=1/2.
\]

For $m$ large enough, $x_m=\sqrt{\mu_m}$ is an asymptotically optimal
strategy for Player~1 in $\Gamma^1_{\mu_m}$, and either $y_m=2\sqrt
{\mu_m}$ or $y'_m=1/2 \sqrt{\mu_m}$ is an optimal strategy for
Player~2 in $\Gamma^1_{\mu_m}$. We have
\[
\lim_{m \rightarrow+\infty} \gamma^1_{\mu_m}(x_m,y_m)=
\lim_{m
\rightarrow+\infty} \gamma^1_{\mu_m}
\bigl(x_m,y'_m\bigr)=5/9.
\]
Thus,
$ \lim_{m \rightarrow+\infty} v_{\mu_m}(1)=5/9$, and
$(v_{\lambda}(1))$ does not converge.

\section*{Acknowledgments}
This paper was partially written during a visit to the Hausdorff
Research Institute for Mathematics at the University of Bonn in the
occasion of the Trimester Program Stochastic Dynamics in Economics and
Finance. Special thanks are due to the organizers for the amazing
working conditions and the warm atmosphere.

I would like to thank J\'{e}r\^ome Renault and Fabien Gensbittel for
their help in clarifying the proof of the main result of this paper. I
am very grateful to Marco Scarsini for helping me to improve the
general presentation of this paper.

I also thank Guillaume Vigeral, Xavier Venel and Sylvain Sorin for
their interesting suggestions.





\printaddresses
\end{document}